\documentclass[leqno, 12pt]{article}
\usepackage{amsmath,amsfonts,amsthm,amssymb,mathrsfs,indentfirst,mathtools,enumerate}
\usepackage[pagebackref,colorlinks]{hyperref}
\hypersetup{linkcolor=blue, urlcolor=blue, citecolor=red}
\usepackage[capitalise]{cleveref}

\newtheorem{theorem}{Theorem}
\newtheorem{lemma}[theorem]{Lemma}
\newtheorem{definition}[theorem]{Definition}
\newtheorem{corollary}[theorem]{Corollary}
\newtheorem{proposition}[theorem]{Proposition}

\newtheorem{problem}[theorem]{Problem}

\theoremstyle{definition}

\newtheorem{remark}[theorem]{Remark}

\newcommand{\R}{\mathbb{R}}
\newcommand{\Q}{\mathbb{Q}}
\newcommand{\Z}{\mathbb{Z}}
\newcommand{\N}{\mathbb{N}}

\begin{document}

\title{Factorization in Finitely-Presented Monoids}

\author{Alfred Geroldinger and Zachary Mesyan}

\maketitle

\begin{abstract}

We study arithmetic properties of factorizations of elements into products of generators, in monoids given with explicit presentations. After relating and comparing this perspective to the more usual approach of factoring into products of atoms, as well as other more recent alternatives, we explore how the relations in the presentation of a monoid affect factorization. In the process, we construct a large class of non-commutative fully elastic monoids. We also show that any finitely-presented cancellative normalizing monoid satisfies the Structure Theorem for Unions. Examples are constructed to demonstrate the sharpness of our results, and exhibit unusual factorization behavior.

\medskip


\end{abstract}

\section{Introduction}

For a long time Factorization Theory was mainly focused on commutative integral domains and commutative cancellative monoids, and the relevant building blocks were \emph{atoms} (i.e., non-units that cannot be expressed as products of non-units). In the last decade the scope has spread rapidly to new algebraic fields, and now includes commutative rings with zero-divisors, as well as non-commutative rings and monoids (with or without idempotents). This necessitated new building blocks, beyond atoms (see~\cite{S, Ge-Zh20a, Co-Go25a, Co25a, Ge-Lo-Ki26} for recent surveys; further specific references are given below).

Transfer homomorphisms are a  key tool in Factorization Theory. They preserve systems of factorization length sets, and they can provide a bridge between objects having quite different algebraic properties. Deep work has revealed that there are transfer homomorphisms from various classes of non-commutative prime rings to commutative monoids \cite{Sm13a, BS, Sm19a}, from non-integrally closed domains to integrally closed domains \cite{Ra25b}, and more \cite{Ba-Po25a}. Nevertheless, such statements are far from being true in general. There are no length-preserving homomorphisms from Weyl algebras to commutative monoids \cite[Lemma 8.1]{BBNS}, or from general power monoids to cancellative monoids \cite[Proposition 4.12]{Fa-Tr18a}. There are finitely-generated right-cancellative monoids, presented by just one relation, whose systems of length sets are distinct from the system of length sets of any commutative finitely-generated cancellative monoid \cite[Corollary 4.4]{GS}. Every finitely-generated cancellative monoid, that is commutative, has \emph{bounded factorization} (i.e., all length sets are finite) and satisfies the \emph{Structure Theorem for Length Sets}. Non-commutative finitely-presented cancellative monoids need not even have bounded factorization \cite[Section 7]{BBNS}. It is well-known and easy to see that commutative noetherian domains have bounded factorization, but it is an open problem whether this remains true for non-commutative noetherian prime rings. These results demonstrate, not surprisingly, that arithmetic investigations of non-commutative monoids and rings need new approaches and  techniques, beyond what was  used in the commutative setting. Furthermore, the loss of commutativity, in classes such as finitely-generated cancellative monoids or noetherian rings, may change arithmetic properties substantially. 

Until quite recently, in the literature on factorizations contained only a small number of papers where factorization invariants are not defined via atoms (e.g.,~\cite{H-K, CDHK}). However, in an effort to generalize the theory and overcome some of the aforementioned hurdles, factorization building blocks have been subsequently adopted to other settings, such as radical ideals, valuation elements, and others (e.g.,~\cite{Ch-Re20a,GS25b, Li-Sc-Si25, Ch-Re26a, Iz-Kn26a}). Moreover, a general preorder-based theory of factorization has been rapidly developing, where a wide array of elements can be used as building blocks (see~\cite{Tr22a, Co-Tr23a, CT, Co-Tr23b, Co-Tr24b}).

This note is devoted to factorizations in not-necessarily-commutative monoids presented by generators and relations--a line of inquiry started in~\cite{GS}, with atoms as building blocks. (For earlier contributions in the commutative cancellative setting see~\cite{C-G-L-M-S12, Ga-Oj-SN13a}). Here, however, instead of considering factorizations of elements into products of atoms (or other aforementioned alternatives), we use generators, since they provide a natural factorization basis in the presence of an explicit presentation. We reformulate the usual arithmetic concepts of Factorization Theory, such as \emph{length} and \emph{elasticity}, in terms of generators (Section~\ref{prelim-sect}), translate various standard results about elasticity and the \emph{Structure Theorem for Unions} to that setting (using abstract results of Tringali~\cite{T} about subadditive collections of natural numbers), and then conduct a systematic study. 

In Section~\ref{gen-sect} we connect factorization in terms of generators to factorization in terms of atoms, as well as other types of building blocks appearing in recent literature, and discuss the consequences of our framework. In particular, generators and atoms are very closely related, and for a \emph{reduced} monoid (i.e., one having no non-identity units) with an irredundant generating set, being atomic is equivalent to the generators and atoms coinciding (Proposition~\ref{atomic-prop}). Then, in Section~\ref{basics-sect}, we explore the manner in which the relations in the presentation of a monoid affect factorization. This aspect was largely neglected in the past, even in the setting of commutative and cancellative monoids (though some steps in that direction have been taken by Chapman, Garc{\'i}a-S{\'a}nchez, Philipp, et al.--see \cite{CGLPR,P1,P2}). Among other observations, we show that in a monoid presented by one relation, the factorization length sets form arithmetic progressions (Proposition~\ref{one-relation-prop}). Moreover, any monoid with a free generator and accepted elasticity is \emph{fully elastic} (Theorem~\ref{fully-elastic-theorem}).

Section~\ref{duo-sect} is devoted to \emph{normalizing} monoids, that is, monoids $M$ such that $aM = Ma$ for all $a \in M$, which constitute a wide and well-studied class (e.g., \cite{JO,Ga-Li11b,G0,BS}). In Theorem~\ref{duo-thrm} we show that a finitely-presented normalizing cancellative monoid satisfies the Structure Theorem for Unions (roughly speaking, the length sets form arithmetic progressions, with very confined exceptions). Moreover, if such a monoid has bounded factorization, then it has accepted elasticity (Proposition~\ref{duo-prop}) and satisfies the \emph{Strong Structure Theorem for Unions}.

In Section~\ref{eg-sect} we give explicit constructions of finitely-presented cancellative monoids that have various pathological factorization behaviors, to demonstrate the sharpness of our results from Section~\ref{duo-sect}. More specifically, in Proposition~\ref{elasticity-eg} we exhibit a one-relation cancellative bounded-factorization monoid that does not have accepted elasticity, but satisfies the Structure Theorem for Unions. Then, in Proposition~\ref{no-structure-eg}, we construct a two-relation cancellative bounded-factorization monoid that does not satisfy the Structure Theorem for Unions. There are few monoids in the literature, of any flavor, that fail to satisfy that theorem, with the first appearing in~\cite{FGKT}. 

\section{Preliminaries} \label{prelim-sect}

In this section we record the relevant definitions and results from the literature, adapted to our setting.

\subsection{General Terminology}

Let $M$ be a (multiplicative) monoid. We denote by $M^{\times}$ the group of invertible elements of $M$, and say that $M$ is \emph{reduced} if $M^{\times} = \{1\}$. An element $a \in M$ is an \emph{atom} if $a \notin M^{\times}$, and for all $b,c \in M$, $a=bc$ implies that either $b \in M^{\times}$ or $c \in M^{\times}$. We denote the set of atoms of $M$ by $\mathcal{A}(M)$, and say that $M$ is \emph{atomic} if every element of $M \setminus M^{\times}$ is a product of atoms. The monoid $M$ is 
\begin{itemize}
\item \emph{cancellative} if for all $a,b,c \in M$, each of $ab=ac$ and $ba=ca$ implies that $b=c$;
\item \emph{unit-cancellative} if for all $a,b \in M$, each of $a=ab$ and $a=ba$ implies that $b \in M^{\times}$;
\item \emph{acyclic} if for all $a,b,c \in M$, $a=bac$ implies that $b, c \in M^{\times}$.
\end{itemize}
Clearly, every cancellative or acyclic monoid is unit-cancellative, but not conversely. A  cancellative non-commutative monoid need not be acyclic. The concept of unit-cancellativity was first introduced in the commutative setting in~\cite{Ro-GS-GG04b}, although under a different name, and independently in~\cite{FGKT}. We would be remiss not to mention 
that the monoid-theoretic notion of unit-cancellativity is reminiscent of Bouvier's notion of \emph{présimplifiable} for commutative rings (see~\cite{MR297758,MR297745}).
The concept of acyclicity was introduced in~\cite{Tr22a}.

Given a set $X$, let $\langle X \rangle$ denote the free monoid generated by $X$, with multiplicative identity $1$.  For a set $X$ and a relation $R \subseteq \langle X \rangle \times \langle X \rangle$ on $\langle X \rangle$, we denote by $\langle X \mid R \rangle$ the monoid generated by $X$ subject to $R$. A monoid $M$ is \emph{finitely-presented} if $M = \langle X \mid R \rangle$, for some finite $X$ and $R$. When specifying the elements of a presentation explicitly, we shall typically suppress set notation, and, for example, write $\langle x,y \mid (x,y) \rangle$, in place of $\langle \{x,y\} \mid \{(x,y)\} \rangle$. From now on, when writing $\langle X \mid R \rangle$, without specifying $X$ and $R$, we shall mean the monoid generated by arbitrary sets $X$ and $R \subseteq \langle X \rangle \times \langle X \rangle$.

Let $M := \langle X \mid R \rangle$. We say that  $X$ is \emph{irredundant} (\emph{for} $M$) if there is no proper subset of $X$ that generates $M$ as a monoid. The relation $R$ is \emph{symmetric} in case $(a,b) \in R$ implies that $(b,a) \in R$, for all $a,b \in R$. Given $a,b \in \langle X \rangle$, we say that $b$ is a \emph{subword} of $a$, or that $a$ \emph{contains} $b$, if $a = cbd$, as elements of $\langle X \rangle$, for some $c,d \in \langle X \rangle$. For all $a,b \in \langle X \rangle$, we write $a=_M b$ if $a$ and $b$ reduce to the same element of $M$. 

We denote by $\N$ (the set of) the natural numbers (including $0$), by $\Z$ the integers, by $\Q$ the rationals, and by $\R$ the real numbers. By $\N^+$ (or $\Z^+$), $\Q^+$, and $\R^+$ we denote the corresponding subsets of positive numbers. 

\subsection{Subadditive Families} \label{subadd-sect}

Next we recall certain arithmetic concepts defined for sets of natural numbers, which will subsequently help us describe factorizations in monoids.

Let $\mathcal{L}$ be a collection of subsets of $\N$. We say that $\mathcal{L}$ is \emph{subadditive} if for all $L_1, L_2 \in \mathcal{L}$, there exists $L_3 \in \mathcal{L}$ such that $L_1 + L_2 \subseteq L_3$, and that $\mathcal{L}$ is \emph{primitive} if
\[
\bigcup_{L \in \mathcal{L}} L \cap \N^+ \neq \emptyset \quad \text{and} \quad \gcd\bigg(\bigcup_{L \in \mathcal{L}} L \cap \N^+\bigg) = 1.
\] 
Also, $\mathcal{L}$ is \emph{directed} if it is subadditive and $1 \in L$, for some $L \in \mathcal{L}$. We note that if $\mathcal{L}$ is directed, then it is necessarily primitive.

For all $k \in \N$, let 
\[
\mathcal{U}_k(\mathcal{L}) := \bigcup \{L \in \mathcal{L} \mid k \in L\},
\]
and, provided that $\mathcal{U}_k(\mathcal{L}) \neq \emptyset$, let
\[
\rho_k(\mathcal{L}) := \sup (\mathcal{U}_k(\mathcal{L})) \in \N \cup \{\infty\} \quad \text{(the \emph{$k$th elasticity of $\mathcal{L}$})},
\]
\[
\lambda_k(\mathcal{L}) := \min (\mathcal{U}_k(\mathcal{L})) \in \N.
\]

Given $L \in \mathcal{L}$, we say that $d \in \N^+$ is a \emph{distance of $L$} if $[k,k+d] \cap L = \{k,k+d\}$ for some $k \in L$, where 
\[
[n,m] := \{l \in \Z \mid n\leq l \leq m\},
\] 
for all $n,m \in \Z$. For each $L \in \mathcal{L}$, let $\Delta(L)$ denote the set of all distances of $L$, let
\[
\rho(L) := \frac{\sup (L \cap \N^+)}{\min (L\cap \N^+)} \in \Q^+ \cup \{\infty\}
\]
in case $L \cap \N^+ \neq \emptyset$, and let $\rho(L) := 0$ otherwise. Also, let
\[
\Delta(\mathcal{L}) := \bigcup_{L \in \mathcal{L}} \Delta (L) \quad \text{(the \emph{set of distances of $\mathcal{L}$})},
\]
and
\[
\rho(\mathcal{L}) := \sup \{\rho(L) \mid L \in \mathcal{L}\} \in \R^+ \cup \{0, \infty\} \quad \text{(the \emph{elasticity of $\mathcal{L}$})}.
\]
We say that $\mathcal{L}$ has \emph{accepted elasticity} if either $\mathcal{L} = \emptyset$, or $\rho(\mathcal{L}) = \rho(L)< \infty$ for some $L \in \mathcal{L}$. Moreover, $\mathcal{L}$ is \emph{fully elastic} if for every $q \in \Q$ satisfying $1 < q < \rho (\mathcal{L})$, there exists $L \in \mathcal{L}$ such that $\rho (L) = q$.

\subsection{Factorizations} \label{factoriz-sect}

Next we define, for each monoid with a presentation, the usual arithmetic indicators of factorization used in the atomic monoid literature, but in terms of generators rather than atoms. It should be noted that our arithmetic indicators very much depend on the specific presentation for the monoid. See Section~\ref{gen-sect} for a more detailed discussion.

Let $M := \langle X \mid R \rangle$. For each $a \in \langle X \rangle$, let $|a| \in \N$ denote the length of $a$, viewed as a word in the elements of $X$, where we understand that $|1| = 0$. For each $a \in \langle X \rangle$, let 
\[
\mathsf{Z}_M(a) := \{b \in \langle X \rangle \mid a=_M b\} \quad \text{(the \emph{set of factorizations of $a$})},
\]
and
\[
\mathsf{L}_M(a) := \{|b| \mid b \in \mathsf{Z}_M(a)\} \quad \text{(the \emph{length set of $a$})}.
\]
We say that $M$ is a \emph{bounded-factorization}, or \emph{BF}, monoid, respectively is \emph{half-factorial}, if $\mathsf{L}_M(a)$ is finite, respectively $\mathsf{L}_M(a) = \{|a|\}$, for all $a \in \langle X \rangle$.

Next, let 
\[
\mathcal{L}(M) := \{\mathsf{L}_M(a) \mid a \in \langle X \rangle\} \quad \text{(the \emph{system of sets of lengths of $M$})}.
\]
It is easy to see that for all $a,b \in \langle X \rangle$, 
\[
\mathsf{L}_M(a)+\mathsf{L}_M(b) \subseteq \mathsf{L}_M(ab),
\]
and hence $\mathcal{L}(M)$ is a subadditive collection of subsets of $\N$. Moreover, $\mathcal{L}(M)$ is directed, and hence also primitive, provided that $X \neq \emptyset$, since in that case $1 \in \mathsf{L}_M(x)$ for any $x\in X$.

Since $\mathcal{L}(M)$ consists of subsets of $\N$, we can apply the arithmetic concepts defined in Section~\ref{subadd-sect} to it. For notational convenience, we define (for all $k \in \N$) $\mathcal{U}_k(M):= \mathcal{U}_k(\mathcal{L}(M))$, $\rho_k(M) := \rho_k(\mathcal{L}(M))$, $\lambda_k(M) := \lambda_k(\mathcal{L}(M))$, $\Delta(M) := \Delta(\mathcal{L}(M))$, and $\rho(M) := \rho(\mathcal{L}(M))$. Likewise, we say that $M$ has \emph{accepted elasticity}, respectively is \emph{fully elastic}, in case $\mathcal{L}(M)$ has that property. 

\subsection{Structure Theorem for Unions}

Let $M := \langle X \mid R \rangle$. We say that $M$ satisfies the \emph{Structure Theorem for Unions} if there exist $d,k^*\in \N^+$ and $m\in \N$, such that for all $k \geq k^*$,
\[
(k + d\Z) \cap [\lambda_{k}(M) + m, \rho_{k}(M) - m] \subseteq \mathcal{U}_k (M) \subseteq k + d\Z.
\]
(Here we understand $[\lambda_{k}(M) + m, \rho_{k}(M) - m]$ to be $\{l \in \Z \mid \lambda_{k}(M) + m \leq l\}$, in case $\rho_{k}(M) = \infty$.) Moreover, $M$ satisfies the \emph{Strong Structure Theorem for Unions} if $M$ satisfies the Structure Theorem for Unions, and there exist $n,k^* \in \N^+$ such that the following hold for all $m \in \N$ and $k \geq k^*:$
\begin{enumerate}[\rm (1)]
\item $(\rho_{k}(M) - \mathcal{U}_{k}(M)) \cap [0,m] = (\rho_{k+n}(M) - \mathcal{U}_{k+n}(M)) \cap [0,m],$
\item $(\mathcal{U}_{k}(M) - \lambda_{k}(M)) \cap [0,m] = (\mathcal{U}_{k+n}(M) - \lambda_{k+n}(M)) \cap [0,m]$.
\end{enumerate}

Since, as noted in Section~\ref{factoriz-sect}, $\mathcal{L}(M)$ is a subadditive primitive collection of subsets of $\N$, provided that $X \neq \emptyset$, we can apply general results of Tringali~\cite{T} about such collections, to obtain statements regarding factorization in $M$, analogous to standard statements of that sort for atomic monoids. We shall use the following on normalizing monoids in Section~\ref{duo-sect}.

\begin{proposition} \label{tringali-prop}
Let $M := \langle X \mid R \rangle$, where $X \neq \emptyset$.
\begin{enumerate}[\rm (1)]
\item If $\Delta(M)$ is finite and $\rho_n(M) = \infty$ for some $n \in \N^+$, then $M$ satisfies the Structure Theorem for Unions.

\item If $\Delta(M)$ is finite, and  there exists $m \in \N^+$ such that $\rho_k(M) - \rho_{k-1}(M) \leq m$ for all sufficiently large $k$, then $M$ satisfies the Structure Theorem for Unions.

\item If $M$ has accepted elasticity, then $M$ satisfies the Strong Structure Theorem for Unions.
\end{enumerate}
\end{proposition}

\begin{proof}
The three statements above follow from~\cite[Theorem 2.20]{T},  \cite[Corollary 2.24]{T}, and~\cite[Theorem 2.27]{T}, respectively. 
\end{proof}

Many other results from~\cite{T}, about subadditive primitive collections, can be likewise translated into statements about monoids with presentations. We record one more, as an illustration.

\begin{proposition}
Let $M := \langle X \mid R \rangle$, where $X \neq \emptyset$, and $\rho(M) < \infty$. Then the following statements are equivalent.
\begin{enumerate}[\rm (1)]
\item $M$ has accepted elasticity.

\item There exists $n \in \N^+$ such that $kn\rho(M) = \rho_{kn}(M)$, for all $k \in \N^+$.

\item There exists $n \in \N^+$ such that $n\rho(M) = \rho_{n}(M)$.
\end{enumerate}
\end{proposition}

\begin{proof}
Since $X \neq \emptyset$, necessarily $\rho(M) \neq 0$, and so $\mathcal{L}(M)$ satisfies the hypotheses of~\cite[Proposition 2.12]{T}. The equivalence of the above statements now follows from that result.
\end{proof}

\section{Generators vs.\ Atoms and Quarks} \label{gen-sect}

As discussed in the Introduction, Factorization Theory has been traditionally developed in the context of (reduced) atomic monoids, but, in a monoid with an explicit presentation, factoring into generators seems more natural. The two perspectives, however, can be reconciled, under reasonable hypotheses. Specifically, we show next that atoms and generators are very closely related, and, provided that the set of generators for a (reduced) monoid is irredundant (which can always be arranged for a finitely-generated monoid, for example), being atomic amounts to those two types of elements coinciding.

\begin{proposition} \label{atomic-prop}
Let $M := \langle X \mid R \rangle$, and let $\Lambda (Y) := \{uyv \mid y \in Y; u,v \in M^{\times}\}$, for all $Y \subseteq X$.
\begin{enumerate}[\rm (1)]
\item $\mathcal{A}(M) = \Lambda (\mathcal{A}(M) \cap X)$. In particular, if $M$ is reduced, then $\mathcal{A}(M) \subseteq X$.

\item If $X \subseteq \mathcal{A}(M)$, equivalently $\Lambda (X) = \mathcal{A}(M)$, then $M$ is atomic.

\item Suppose that $M$ is reduced and $X$ is irredundant. Then $M$ is atomic if and only if $\mathcal{A}(M) = X$.
\end{enumerate}
\end{proposition}

\begin{proof}
(1) Suppose that $a \in \mathcal{A}(M)$. Since $a \neq 1$, we can write $a = x_1\cdots x_n$ for some $x_1, \dots, x_n \in X$ and $n \in \N^+$. Let $m \in \{1, \dots, n\}$ be the minimal index such that $x_m \notin M^{\times}$ (which must exist, since $a \notin M^{\times}$), and let $b := x_1 \cdots x_{m-1}$ and $c := x_{m+1} \cdots x_n$ (where $b$, respectively $c$, is understood to be $1$ in case $m=1$, respectively, $m=n$). Then $b \in M^{\times}$, and $b x_{m} \notin M^{\times}$. Since $a= (bx_{m})c$, and $a \in \mathcal{A}(M)$, it must be the case that $c \in M^{\times}$. Thus $a \in \Lambda (X)$, and so $\mathcal{A}(M) \subseteq \Lambda (X)$. 

Next, it is easy to see that for all $a \in M$ and $u,v \in M^{\times}$, we have $a \in \mathcal{A}(M)$ if and only if $uav \in \mathcal{A}(M)$. (Specifically, if $a \in \mathcal{A}(M)$, $u,v \in M^{\times}$, and $b,c \in M$ are such that $uav =_M bc$, then $a =_M (u^{-1}b)(cv^{-1})$, and so either $b \in M^{\times}$ or $c \in M^{\times}$, which shows that $uav \in \mathcal{A}(M)$. The converse follows by symmetry.) In particular, for all $x \in X$ and $u,v \in M^{\times}$, we have $x \in \mathcal{A}(M)$ if and only if $uxv \in \mathcal{A}(M)$. Since $\mathcal{A}(M) \subseteq \Lambda (X)$, we conclude that $\mathcal{A}(M) = \Lambda (\mathcal{A}(M) \cap X)$.

If $M$ is reduced, then $\Lambda (X) = X$, from which the second claim follows.

\smallskip

(2) If $X \subseteq \mathcal{A}(M)$, then every nonidentity element of $M$ can be expressed as a product of atoms, and therefore $M$ is atomic. Also, clearly, if $\Lambda (X) = \mathcal{A}(M)$, then $X \subseteq \mathcal{A}(M)$. Finally, as shown above, if $X \subseteq \mathcal{A}(M)$, then $\Lambda (X) \subseteq \mathcal{A}(M)$, and hence $\Lambda (X) = \mathcal{A}(M)$, by (1).

\smallskip

(3) If $\mathcal{A}(M) = X$, then $M$ is atomic, by (2). Conversely, suppose that $M$ is atomic. Suppose further that there exists $x \in X \setminus \mathcal{A}(M)$. Since $X$ is irredundant, $x \neq_M 1$, and therefore $x \notin M^{\times}$, given that $M$ is reduced. Thus $x$ can be expressed as a product of elements of $\mathcal{A}(M)$, necessarily distinct from $x$. But then, since $M$ is reduced, (1) implies that $X \setminus \{x\}$ generates $M$ as a monoid, contradicting the hypothesis that $X$ is irredundant. Therefore $\mathcal{A}(M) = X$, again, by (1). 
\end{proof}

Considering factorizations into generators, rather than atoms, has the advantage of being applicable to arbitrary monoids, even if they have no atoms at all (e.g., antimatter domains~\cite{A-C-H-Z07, Be-El23a}). Moreover, generator-based factorization can be viewed as a general platform for exploring factorization into building blocks of various sorts, as we describe more explicitly below.

\begin{remark} \label{gen-remark}
Let $M$ be a monoid, let $X \subseteq M \setminus \{1\}$, and let $N$ be the submonoid of $M$ generated by $X$. Then $N \cong \langle X \mid \ker(\phi) \rangle$, where $\phi : \langle X \rangle \to N$ is the homomorphism induced by sending each element of $X$ to the corresponding copy in $N$, and $\phi(1)=1$. Hence taking $X$ to be the set of building blocks for a theory of factorization in $M$ (e.g., primes, atoms, irreducible elements), one can reformulate that theory in terms of generators, upon restricting to the submonoid $N$ consisting of the elements of $M$ that can be factored into products of the relevant building blocks.
\end{remark}

The essence of the idea of restricting to the submonoid generated by factorization building blocks is commonly used in, for example, the Factorization Theory of commutative non-atomic monoids, where factorizations are considered only for elements that can be written as products of atoms. Likewise, using presentations consisting of factorization building blocks is not new. Specifically, the \emph{factorization monoid}, which is the free commutative monoid generated by the atoms of a given commutative monoid, is a commonly-used tool in the factorization literature (see, e.g.,~\cite{P1,P2,BS,FGKT,Ge-Zh20a,GK}). Rather, our framework is a natural development of the earlier approaches, that is particularly well-suited to working in the setting where one typically constructs examples, namely monoids with explicit presentations.

A drawback of factorization into generators, as noted in Section~\ref{factoriz-sect}, is that the various factorization-related notions we have defined depend on the specific presentation of a given monoid. To take an extreme example, the trivial monoid $\{1\}$ is BF with respect to the standard presentation $\langle \emptyset \rangle$, but not with respect to the presentation $\langle x \mid (1,x) \rangle$. Moreover, as with atoms, and other types of factorization building blocks, the generator-based arithmetic indicators are not always informative. For example, suppose that $M := \langle X \mid R \rangle$ is such that $ab =_M 1$, for some $a,b \in \langle X \rangle \setminus \{1\}$. Then for all $c \in \langle X \rangle$ and  $n \in \N$, we have
\[|c| + n|ab| = |c(ab)^n| \in \mathsf{L}_M(c),\]
and so $\mathsf{L}_M(c)$ contains an infinite arithmetic progression with difference $|ab|$. In this situation the system of lengths of $M$ sheds relatively little light on the factorization properties of $M$. Similar problems arise, for example, in semigroups with nontrivial idempotents.

We tame problematic behavior of this sort in many of our results by choosing generating sets with desirable features, such as being irredundant. A more conventional approach is to impose some sort of regularity condition on the monoid, rather than the generators. We do precisely that in Section~\ref{duo-sect}, where we assume, in most of the results, that our monoids are BF. The next lemma, versions of which can be found in various other settings (e.g., \cite[Corollary 2.29]{Fa-Tr18a}, \cite[Example 5.4]{Co-Tr23a}, \cite[Lemma 2.2 and Theorem 2.3]{BBNS}), explains some of the consequences of that assumption.

\begin{lemma} \label{unit-canc-lem}
Let $M := \langle  X \mid  R \rangle$, and suppose that $M$ is BF. If $a =_M bac$ for some $a,b,c \in \langle X \rangle$, then $b = 1 = c$. In particular, $M$ is acyclic and reduced.
\end{lemma}

\begin{proof}
Let $a,b,c \in \langle X \rangle$, and suppose that $a =_M bac$. Then $a=_M b^nac^n$, and hence 
\[
|b^nac^n| = |a| + n(|b|+|c|) \in \mathsf{L}_M(a),
\] 
for all $n \in \N$. Given that $M$ is BF, this is possible only if $|b| = 0 = |c|$, that is, $b=1=c$. It follows immediately that $M$ is acyclic.

The above argument, in particular, shows that if $bc =_M 1$, then $b = 1 = c$. Thus, $M$ cannot have any nontrivial invertible elements, and is therefore reduced.
\end{proof}

We conclude this section by relating our setting to the preorder-based theory of factorization developed recently by Cossu and Tringali in~\cite{Tr22a,Co-Tr23a,CT,Co-Tr23b,Co-Tr24b}. Letting $M$ be a monoid equipped with a preorder $\preceq$ (i.e., reflexive transitive relation), one can define factorization building blocks with respect to that preorder, such as $\preceq$-\emph{irreducible} elements, $\preceq$-\emph{atoms}, $\preceq$-\emph{quarks}, and $\preceq$-\emph{prime} elements (see~\cite[Definition 3.6]{Tr22a}). For example, $a \in M$ is a $\preceq$-\emph{quark} if $b \prec a$ implies that $b \preceq 1 \preceq b$, for all $b \in M$, and it is not the case that $a \preceq 1 \preceq a$. Various results are then proved about factoring elements into products of the above types of building blocks, depending on properties of the preorder.

By Remark~\ref{gen-remark}, this theory can be situated within our framework (though that may not necessarily be enlightening), by taking the $\preceq$-irreducible elements, or other types of elements mentioned above, as the generators in the presentation of an appropriate submonoid. On the other hand, our approach can be rephrased, in many cases, using the Cossu/Tringali terminology. More specifically, let $M := \langle X \mid R \rangle$, and for all $a,b \in M$ write $a \preceq b$ if and only if $\min (\mathsf{L} (a)) \leq  \min (\mathsf{L} (b))$. Then $\preceq$ is a preorder on $M$ (which is a \emph{pullback}, in the sense of~\cite[Example 3.3(2)]{Tr22a}, of the usual order on $\N$), and the $\preceq$-quarks are precisely the $a \in M$ such that $a =_M x$, for some $x \in X$. Thus, given a reasonably well-behaved generating set $X$, the generators of $M$ can be identified with its $\preceq$-quarks.

\section{Basic Observations} \label{basics-sect}

This section is devoted to general observations about how the shape of $R$ affects factorizations of the elements of $\langle X \mid R \rangle$. We begin by showing that, in a monoid presented by just one relation, factorization tends to be well-behaved.

\begin{proposition} \label{one-relation-prop}
Let $M := \langle X \mid R \rangle$, where $R=\{(e,f)\}$ for some $e,f \in \langle X \rangle$, and let $d := ||e|-|f|| \in \N$. 
\begin{enumerate}[\rm (1)]
\item If $d=0$, then $M$ is half-factorial.

\item If $d\neq 0$, then $\Delta(M) = \{d\}$, and $\, \mathsf{L}_M(a)$ is an arithmetic progression with difference $d$, for all $a \in \langle X \rangle$.
\end{enumerate}
\end{proposition}

\begin{proof}
(1) Obvious.

\smallskip

(2) Suppose that $d\neq 0$, and let $a \in \langle X \rangle$. If $\mathsf{Z}_M(a) = \{a\}$, then $\mathsf{L}_M(a) = \{|a|\}$ is vacuously an arithmetic progression (with difference $d$). Let us therefore suppose that there exists $b \in \mathsf{Z}_M(a) \setminus \{a\}$. Upon replacing $a$ with a suitable element of $\mathsf{Z}_M(a)$, we may also assume, without loss of generality, that $|a| = \min(\mathsf{L}_M(a))$.

Since $a =_M b$ and $a \neq b$, there are $m \in \N^+$ and $c_0, \dots, c_m \in \langle X \rangle$ such that $c_0 = b$, $c_m = a$, and each $c_i$ is connected to $c_{i+1}$ via an elementary $R$-transition. That is, for each $i<m$, there exist $d_i, g_i, g_i', h_i \in \langle X \rangle$ satisfying $(g_i, g_i') \in \{(e,f), (f,e)\}$, $c_i = d_ig_ih_i$, and $c_{i+1} = d_ig_i'h_i$. (See~\cite[Proposition 1.5.9]{H} for more details.) Thus $||c_{i+1}| - |c_{i}|| = d$ for each $i$, and $\Delta(\mathsf{L}_M(a)) = \{d\}$. Moreover, given that $|a| = \min(\mathsf{L}_M(a))$, and hence $|a| \leq |b|$, we have
\[
\{|a|, |a| + d, |a| + 2d, \dots, |b|\} \subseteq \mathsf{L}_M(a) \cap (|a| + d\N).
\]
Since $b \in \mathsf{Z}_M(a)\setminus \{a\}$ was arbitrary, it follows that $\mathsf{L}_M(a)$ is an arithmetic progression with difference $d$.

Finally, since $\Delta(\mathsf{L}_M(a)) \subseteq \{d\}$ for all $a \in \langle X \rangle$, we conclude that $\Delta(M) = \{d\}$.
\end{proof}

It follows easily from Proposition~\ref{one-relation-prop} that a one-relation monoid must satisfy the Structure Theorem for Unions. A monoid presented by two or more relations, however, need no longer do so, as we show in Proposition~\ref{no-structure-eg}. Note also that the arguments in the proof of Proposition~\ref{one-relation-prop} work equally well for commutative one-relation monoids.

While one-relation monoids are well-behaved, in the above sense, they can exhibit some interesting factorization behaviors none-the-less. Specifically, we show in Proposition~\ref{elasticity-eg} that such a monoid need not have accepted elasticity, even when the elasticity is finite. Additionally, in Lemma~\ref{eg-lem} we construct a family of one-relation monoids $M$ having the unusual property that the differences $\rho_k(M) - \rho_{k-1}(M)$ are finite but unbounded. We also note that a (commutative) one-relation monoid need not be fully elastic (cf.\ Theorem~\ref{fully-elastic-theorem} below). Specifically, if $n_1, n_2 \in \N$ satisfy $1 < n_1 < n_2$ and $\gcd (n_1, n_2)=1$, then the submonoid $M$ of $(\N,+)$ generated by $\{n_1, n_2\} = \mathcal{A}(M)$ is not fully elastic (see \cite[Theorem 5.5]{Ge-Sc-Zh17b} for more details). More generally speaking, one-relation monoids have historically served as a rich source of interesting questions--see, e.g.,~\cite{NB}.

In the next lemma we gather general facts about how the relations in a presentation for a monoid influence various arithmetic indicators of factorization.

\begin{lemma} \label{basics-lemma}
Let $M := \langle X \mid R \rangle$, where $R$ is symmetric, and let $N$ be the additive submonoid of $\Z$ generated by $\, \{|a|-|b| \mid (a,b)\in R\}$.
\begin{enumerate}[\rm (1)]
\item If $a,b \in \langle X \rangle$ satisfy $a=_M b$, then $\, |a|-|b| \in N$.

\item Let $a,b \in \langle X \rangle$ be such that $a=_M b$, $|b| < |a|$, and for all $c \in \mathsf{Z}_M(a)$ it is not the case that $\, |b| < |c| < |a|$. Then there exist $e,f \in \langle X \rangle$ such that $(e,f) \in R$ and $\, |a| - |b| \leq |e|-|f|$.

\item For all $n \in N$, there exist $a,b \in \langle X \rangle$ such that $a=_M b$ and $n = |a|-|b|$.

\item If $N \neq \{0\}$, then $N = d\Z$, where $d = \gcd (N) = \min(\Delta (M)) = \gcd(\Delta (M))$.
\end{enumerate}
\end{lemma}

\begin{proof}
(1) Let $a,b \in \langle X \rangle$ be such that $a=_M b$. If $a = b$, then $|a|-|b| = 0 \in N$, and so we may assume that $a \neq b$. Then there must exist $m \in \N^+$ and $c_0, \dots, c_m \in \langle X \rangle$ such that $c_0 = b$, $c_m = a$, and each $c_i$ is connected to $c_{i+1}$ via an elementary $R$-transition~\cite[Proposition 1.5.9]{H}. Given that $R$ is symmetric, this means that for each $i<m$ there exist $d_i, e_i, e_i', f_i \in \langle X \rangle$, such that $(e_i, e_i') \in R$, $c_i = d_ie_if_i$, and $c_{i+1} = d_ie_i'f_i$. Then 
\[
|a|-|b| = |c_m| - |c_0| = \sum_{i=0}^{m-1} (|c_{i+1}| - |c_i|) = \sum_{i=0}^{m-1} (|e_i'|-|e_i|) \in N.
\]

\smallskip

(2) As above, there exist $m \in \N^+$, $c_0, \dots, c_m \in \mathsf{Z}_M(a) \subseteq \langle X \rangle$, and $d_i, e_i, e_i', f_i \in \langle X \rangle$ (for each $i<m$), such that $c_0 = b$, $c_m = a$, $(e_i, e_i') \in R$, $c_i = d_ie_if_i$, and $c_{i+1} = d_ie_i'f_i$. Since $|b| < |c_m|$, and since for all $i \in \{0, \dots, m\}$, either $|c_i| \leq |b|$ or $|a| \leq |c_i|$, there is a least $i \in \{0, \dots, m-1\}$ such that $|a|\leq |c_{i+1}|$. Then $|c_i| \leq |b|$, and so
\[
|a|-|b| \leq |c_{i+1}|-|c_i| = |d_ie_i'f_i| - |d_ie_if_i| = |e_i'|-|e_i|,
\]
as desired.

\smallskip

(3) Let $n \in N$. Since $0 = |1|-|1|$ (where $1 \in \langle X \rangle$), we may assume that $n \neq 0$. Then $n = \sum_{i=1}^k m_i(|c_i|-|d_i|)$, for some $k,m_i \in \N^+$ and $(c_i,d_i) \in R$. Letting $a := c_1^{m_1} \cdots c_k^{m_k} \in \langle X \rangle$ and $b := d_1^{m_1} \cdots d_k^{m_k} \in \langle X \rangle$, we have $a=_M b$ and $n = |a|-|b|$.

\smallskip

(4) Suppose that $N \neq \{0\}$. Since $R$ is symmetric, $N$ is closed under taking additive inverses, and is therefore a subgroup of $\Z$. Thus, by a standard argument, $N = d\Z$, where $d$ is the minimal element of $N \cap \N^+$, given that $N \neq \{0\}$. Necessarily $d = \gcd(N \cap \N) = \gcd (N)$. 

Since, by (1), $\Delta (M) \subseteq N$, we must have $d \leq \min (\Delta (M))$. Since $d \in N$, by (3), there exist $a,b \in \langle X \rangle$ such that $a=_M b$ and $d = |a|-|b|$, which implies that $\min (\Delta (M)) \leq d$. Thus $d = \min (\Delta (M))$. 

According to~\cite[Proposition 2.9]{FGKT}, $\min(\Delta (\mathcal{L})) = \gcd(\Delta (\mathcal{L}))$, for any directed collection $\mathcal{L}$ of subsets of $\N$. Since $N \neq \{0\}$ implies that $X \neq \emptyset$, the system of lengths $\mathcal{L}(M)$ is indeed directed (see Section~\ref{factoriz-sect}), and so we conclude that $d = \gcd(\Delta (M))$.
\end{proof}

The following is a generalization of~\cite[Proposition 6]{G}.

\begin{corollary} \label{fin-delta}
Let $M := \langle X \mid R \rangle$, where $\, \{|a|-|b| \mid (a,b)\in R\}$ is finite. Then $\Delta(M)$ is also finite.
\end{corollary}

\begin{proof}
Letting $\overline{R}$ denote the symmetric closure of $R$, we have $\langle X \mid R \rangle = \langle X \mid \overline{R} \rangle$. Moreover, $\{|a|-|b| \mid (a,b)\in R\}$ is finite if and only if $\{|a|-|b| \mid (a,b)\in \overline{R}\}$ is finite. So, upon replacing $R$ with $\overline{R}$, we may assume, without loss of generality, that $R$ is symmetric.

Now, let 
\[
m := \max \{|a|-|b| \mid (a,b)\in R\} \in \N.
\] 
Then, by Lemma~\ref{basics-lemma}(2), $n \leq m$ for all $a \in \langle X \rangle$ and $n \in \Delta(\mathsf{L}_M(a))$. Thus $\Delta(M) \subseteq [0, m]$.
\end{proof}

It can be difficult to determine whether the sets $\mathcal{U}_k(M)$ are finite, for a monoid $M:=\langle  X \mid R \rangle$. However, if either $X$ or $R$ is finite, then the $\mathcal{U}_k(M)$ being finite amounts to each element of $M$ having only finitely many distinct factorizations, as we show next.

\begin{proposition} \label{fin-elasticity-prop}
Let $M := \langle X \mid R \rangle$, and suppose that
\[Y := \{x \in X \mid \exists a,b,c \in \langle X \rangle \text{ such that } (axb,c) \in R \text{ or }(c,axb) \in R\}.\]
is finite. Then the following statements are equivalent.
\begin{enumerate}[\rm (1)]
\item There exists $a \in \langle X \rangle$ such that $\, \mathsf{Z}_M(a)$ is infinite.

\item There exists $a \in \langle X \rangle$ such that $\, \mathsf{L}_M(a)$ is infinite.

\item There exists $k \in \N$ such that $\mathcal{U}_k(M)$ $($as well as $\rho_k(M)$$)$ is infinite.

\item There exists $k \in \N$ such that $\mathcal{U}_n(M)$ $($as well as $\rho_k(M)$$)$ is infinite for all $n \geq k$.
\end{enumerate}
\end{proposition}

\begin{proof}
$(1) \Rightarrow (2)$ Suppose that $\mathsf{Z}_M(a)$ is infinite for some $a \in \langle X \rangle$. Since the elements of $X \setminus Y$ do not contribute to producing factorizations in $M$ of different lengths, upon replacing $a$ with a subword, we may assume that $\mathsf{Z}_M(a) \subseteq \langle Y \rangle$. Since $Y$ is finite, there are only finitely many words in $\langle Y \rangle$ of each length, and so $\mathsf{Z}_M(a)$ must contain elements of infinitely many different lengths, making $\mathsf{L}_M(a)$ infinite. 

\smallskip

$(2) \Rightarrow (3)$ If $\mathsf{L}_M(a)$ is infinite for some $a \in \langle X \rangle$, then so are $\mathcal{U}_{|a|}(M)$ and $\rho_{|a|}(M)$. 

\smallskip

$(3) \Rightarrow (4)$ This follows from the fact that if $k \in \N$ and $a \in \langle X \rangle$ are such that $|a|=k$, then 
\[
(n-k) + \mathsf{L}_M (a) \subseteq \mathsf{L}_M (x^{n-k}a) \subseteq \mathcal{U}_n(M)
\]
for all $n \geq k$ and $x \in X$.

\smallskip

$(4) \Rightarrow (1)$ We shall prove the contrapositive. Suppose that $\mathsf{Z}_M(a)$ is finite for all $a \in \langle X \rangle$, and let $k \in \N$. Also, let $\{a_i \mid i \in I\}$ be the set of all elements of $\langle X \rangle$ of length $k$. If $X$ is finite, then so is $I$. In that case $\bigcup_{i \in I} \mathsf{Z}_M(a_i)$, and hence also $\bigcup_{i \in I} \mathsf{L}_M(a_i) = \mathcal{U}_k(M)$, is finite.

So suppose that $X$ is infinite, and let $x \in X \setminus Y$. For each $i \in I$, let $b_i \in \langle Y \cup \{x\} \rangle$ be the word obtained from $a_i$ by replacing every occurrence of every element of $X \setminus Y$ with $x$. Then $|b_i| = |a_i|$ and $\mathsf{L}_M(b_i) = \mathsf{L}_M(a_i)$, for each $i \in I$. Since $Y \cup \{x\}$ is finite, so is $\{b_i \mid i \in I\}$, and hence also $\bigcup_{i \in I} \mathsf{Z}_M(b_i)$. Thus 
\[
\bigcup_{i \in I} \mathsf{L}_M(b_i) =\bigcup_{i \in I} \mathsf{L}_M(a_i) = \mathcal{U}_k(M)
\] 
must be finite as well.
\end{proof}

There is extensive literature (e.g., \cite{Zh19a, GK, Re26a, Bo-Co-Kh26a}) exploring the question of which (atomic) monoids and rings are fully elastic, but so far it has been confined to the commutative setting. In our next result we construct a large class of non-commutative fully elastic monoids.

\begin{theorem} \label{fully-elastic-theorem}
Let $M := \langle X \mid R \rangle$, and let 
\[
Y := \{x \in X \mid \exists a,b,c \in \langle X \rangle \text{ such that } (axb,c) \in R \text{ or }(c,axb) \in R\}.
\]
If $M$ has accepted elasticity, and $X\neq Y$, then $M$ is fully elastic.
\end{theorem}

\begin{proof}
Suppose that $M$ has accepted elasticity, and that there exists $x \in X\setminus Y$. Then, in particular, $X \neq \emptyset$, and so $\rho(M) \geq 1$. If $\rho (M)=1$, then $M$ is fully elastic vacuously, and so we may assume that $\rho (M) > 1$. By the choice of $x$,
\begin{equation} \label{calc3}
\mathsf{L}_M (x^ka) = k + \mathsf{L}_M (a),
\end{equation}
for all $a \in \langle X \rangle$ and $k \in \N$.

Next, for all $a \in \langle X \rangle \setminus \{1\}$ and $k \in \N^+$, we have
\begin{equation} \label{calc4}
\max (\mathsf{L}_M (a^k)) \geq k \max (\mathsf{L}_M(a)) \quad \text{and} \quad \min (\mathsf{L}_M (a^k)) \leq k \min (\mathsf{L}_M (a)),
\end{equation}
which implies that
\begin{equation} \label{calc1}
\rho (\mathsf{L}_M (a^k)) = \frac{\max (\mathsf{L}_M (a^k))}{\min (\mathsf{L}_M (a^k))} \geq \frac{k \max(\mathsf{L}_M(a))}{k \min (\mathsf{L}_M(a))} = \rho (\mathsf{L}_M (a)).
\end{equation}
Since $M$ has accepted elasticity and $\rho(M) \neq 0$, there exists $c \in \langle X \rangle \setminus \{1\}$ such that $\rho (M) = \rho (\mathsf{L}_M(c))$. Then $\rho (\mathsf{L}_M(c^k)) = \rho (\mathsf{L}_M(c))$ for all $k \in \N^+$, and so, using~\eqref{calc4} and~\eqref{calc1}, we conclude that 
\begin{equation} \label{calc2}
\max(\mathsf{L}_M(c^k)) = k \max (\mathsf{L}_M (c)) \quad \text{and} \quad \min (\mathsf{L}_M (c^k)) = k \min (\mathsf{L}_M(c)).
\end{equation}

Finally, let $q \in \mathbb Q$ be such that $1 < q < \rho (M)$, write $q = r/s$ for some $r,s \in \N^+$, and set 
\[
n := r - s, \quad  k := s \max (\mathsf{L}_M(c)) - r \min (\mathsf{L}_M(c)), \quad \text{and} \quad b := x^kc^n.
\]
Then, by~\eqref{calc3} and~\eqref{calc2},
\[
\max (\mathsf{L}_M(b)) = k + n \max(\mathsf{L}_M(c)) \quad \text{and} \quad \min (\mathsf{L}_M(b)) = k + n \min (\mathsf{L}_M(c)).
\]     
Putting it all together, we obtain
\[
\begin{aligned}
\rho (\mathsf{L}_M(b) ) & = \frac{\max (\mathsf{L}_M(b))}{\min (\mathsf{L}_M(b))} = \frac{k + n \max (\mathsf{L}_M(c))}{k + n \min(\mathsf{L}_M (c))} \\
 & = \frac{(r-s)\max (\mathsf{L}_M(c)) + s \max (\mathsf{L}_M(c)) - r \min (\mathsf{L}_M(c))}{(r-s)\min(\mathsf{L}_M(c)) + s \max (\mathsf{L}_M(c)) - r \min (\mathsf{L}_M(c))} \\
 & = \frac{r (\max (\mathsf{L}_M (c)) - \min (\mathsf{L}_M (c)))}{s (\max (\mathsf{L}_M(c)) - \min (\mathsf{L}_M (c)))} \\
 & = \frac{r}{s} = q,
\end{aligned}
\]
and so $M$ is fully elastic.
\end{proof}

\section{Normalizing Monoids} \label{duo-sect}

The main goal of this section is to show that finitely-presented normalizing cancellative monoids satisfy the Structure Theorem for Unions. We begin by recalling notions and basic facts related to normalizing monoids.

\begin{definition}
Let $M$ be a monoid. An element $a \in M$ is called \emph{normal} $($or \emph{invariant}$)$ if $aM = Ma$. We denote by $\, \mathsf{N} (M)$ the set of all normal elements of $M$. If $M = \mathsf{N} (M)$, then $M$ is called \emph{normalizing} $($or \emph{normal} or \emph{duo}$)$.
\end{definition}

It is easy to see that $\mathsf{N} (M)$ is a submonoid of $M$, and $M^{\times} \subseteq \mathsf{N} (M)$, for any monoid $M$. For this reason, $\mathsf{N}(M)$ is referred to as the \emph{normalizing submonoid} of $M$. Given a finitely-generated monoid $M$, the submonoid $\mathsf{N}(M)$ need not be finitely-generated, though that is the case in certain situations--see, e.g., \cite[Corollary 4.4.12]{JO}. Additionally, each normalizing monoid satisfies the left and right Ore conditions (i.e., $aM \cap bM \neq \emptyset$, and $Ma \cap Mb \neq \emptyset$, for all $a,b \in M$), and therefore each cancellative normalizing monoid has a quotient group--see~\cite[Theorems 1.24 and 1.25]{CP} for more details.

There is extensive literature on normal elements in the Ideal Theory and Factorization Theory of monoids and rings--see, e.g., \cite{AM,O,S,Co-Tr23a}. Normalizing \emph{Krull} monoids, in particular, have received attention within the factorization context~\cite{G0,BS}. 

Let us record another interesting fact about normal elements, which has been previously noted in various other settings (see \cite{G0, GS-Tr25a}).

\begin{lemma}
Let $M$ be a monoid that is contained in a group. Then $aM^{\times} = M^{\times}a$, for all $a \in \mathsf{N} (M)$.
\end{lemma}

\begin{proof}
Let $a \in \mathsf{N} (M)$ and $u \in M^{\times}$. Then $au \in aM = Ma$, and so $au a^{-1} \in M$, since $M$ is contained in a group. Likewise $au^{-1}a^{-1} \in M$, and hence $au a^{-1} \in M^{\times}$. Thus $au = (aua^{-1})a \in M^{\times}a$, and therefore $aM^{\times} \subseteq M^{\times} a$. By symmetry, $aM^{\times} = M^{\times}a$.
\end{proof}

We now turn to normalizing monoids with presentations.

\begin{lemma} \label{duo-lem}
Let $M := \langle X \mid R \rangle$, where $X$ is irredundant, and suppose that $M$ is normalizing, cancellative, and BF. 
\begin{enumerate}[\rm (1)]
\item For all $x,y \in X$, there exists $z \in X$ such that $xy=_M zx$.

\item Suppose that $X$ is finite but nonempty, and write $X = \{x_{1}, \dots, x_{n}\}$, for some $n \in \N^+$. Then for each $a \in \langle X \rangle$, there is a $($necessarily unique$)$ $\nu(a) \in \langle X \rangle$ such that $a =_M \nu(a)$, $|a| = |\nu(a)|$, $\nu(a) = x_{1}^{m_1} \cdots x_{n}^{m_n}$, each $m_i \in \N$, and $(m_1, \dots, m_n)$ is minimal in the lexicographic ordering on $\, \N^n$ among $n$-tuples of that form.
\end{enumerate}
\end{lemma}

\begin{proof}
(1) Suppose that $x,y \in X$. Since $M$ is normalizing, $xy =_M ax$ for some $a \in \langle X \rangle$. It cannot be the case that $a =_M 1$, since then, using cancellativity, we would have $y =_M 1$, contrary to the hypothesis that $X$ is irredundant (and also the hypothesis that $M$ is BF). Hence $a \neq 1$ as an element of $\langle X \rangle$, and so we can write $a = bz$ for some $b \in \langle X \rangle$ and $z \in X$. 

Seeking a contradiction, suppose that $b \neq_M 1$. Since $M$ is normalizing, $zx =_M xc$ for some $c \in \langle X \rangle$, and so $xy =_M bzx =_M bxc$. As before, $M$ being cancellative and $X$ being irredundant imply that $c \neq 1$. Using the fact that $M$ is normalizing once more, $bx =_M xd$ for some $d \in \langle X \rangle$, where, necessarily, $d \neq 1$, since $b \neq_M 1$. Thus $xy=_M xdc$, which gives $y =_M dc$, by cancellativity. Since $c \neq 1$ and $d \neq 1$, we have $|dc| \geq 2$. Hence, by Lemma~\ref{unit-canc-lem}, $dc$ does not contain any instances of $y$, as a word. But then $y =_M dc$ contradicts $X$ being irredundant. Hence $b =_M 1$, and so $xy =_M zx$.

\smallskip

(2) We proceed by induction on word length. If $a \in \langle X \rangle$ is such that $|a| = 0$, then $a = 1$, and $\nu(1) = x_{1}^{0} \cdots x_{n}^{0}$ clearly satisfies the properties in the statement. Let us now assume inductively that there exists $k \in \N$ such that the statement holds for all $a \in \langle X \rangle$ satisfying $|a| \leq k$.

Let $a \in \langle X \rangle$ be such that $|a| = k+1$, and let 
\[
S := \{b \in \mathsf{Z}_M(a) \mid k+1 = |b|\}.
\]
Since $X$ is finite, so is $S$. We may therefore choose $b \in S$ such that $b$ contains $x_j$ with $j \in \{1, \dots, n\}$ maximal among the elements of $S$, and such that the number $m_j \in \N^+$ of instances of $x_j$ is maximal among those elements. Using (1) repeatedly, we can find $c \in \langle X \rangle$ such that $a =_M b =_M cx_j^{m_j}$ and $|c| + m_j = |a| = k+1$. Then $|c| \leq k$, and so using the inductive hypothesis, in light of the choice of $b$, we have $\nu(c) = x_{1}^{m_1} \cdots x_{j-1}^{m_{j-1}}x_{j}^{0} \cdots x_{n}^{0}$, where $c =_M \nu(c)$, $|c| = |\nu(c)|$, each $m_i \in \N$, and $(m_1, \dots, m_{j-1}, 0, \dots, 0)$ is minimal in the lexicographic ordering on $\N^n$ among $n$-tuples of that form. Then it follows easily, using the cancellativity of $M$, that 
\[
\nu(a) :=  x_{1}^{m_1} \cdots x_{j-1}^{m_{j-1}}x_j^{m_j}x_{j+1}^{0} \cdots x_{n}^{0}
\] 
has the desired properties.
\end{proof}

Our next result is an analogue of the well-known fact that every atomic finitely-generated commutative cancellative monoid has accepted elasticity (in terms of atoms)--see, e.g.~\cite[Remark 3.11(5)]{FGKT}. Note that unlike the atomic commutative setting, we must impose the assumption that our monoid is BF, to exclude, for example, abelian groups--see the discussion in Section~\ref{gen-sect} for more details.

\begin{proposition} \label{duo-prop}
Let $M := \langle X \mid R \rangle$, where $X$ is finite, and suppose that $M$ is normalizing, cancellative, and BF. Then $M$ has accepted elasticity.
\end{proposition}

\begin{proof}
Since $X$ is finite, upon replacing it with a suitable subset, we may assume that $X$ is irredundant. If $X = \emptyset$, then $\rho(M) = \rho(\mathsf{L}_M(1)) = 0$. So we may assume that $X$ is nonempty, and write $X = \{x_{1}, \dots, x_{n}\}$, for some $n \in \N^+$. 

For each $a \in \langle X \rangle$, let $\nu(a) = x_{1}^{k_1} \cdots x_{n}^{k_n} \in \langle X \rangle$ denote the unique representation of $a$ given in Lemma~\ref{duo-lem}(2), and let
\[
S := \{(\nu(a),\nu(b)) \in \langle X \rangle \times \langle X \rangle \mid  a =_M b\} \setminus \{(x_{1}^{0} \cdots x_{n}^{0},x_{1}^{0} \cdots x_{n}^{0})\}.
\]
Define a function $f:S \to \N^{2n}$ via 
\[
f((x_{1}^{k_1} \cdots x_{n}^{k_n},x_{1}^{m_1} \cdots x_{n}^{m_n})) := (k_1, \dots, k_n,m_1, \dots, m_n).
\]
By Dickson’s Theorem~\cite[Theorem 1.5.3]{GHK}, there are only finitely many minimal points in the image $f(S)$ of $f$, relative to the usual product order on $\N^{2n}$ (i.e., $(p_{1},p_{2},\dots, p_{2n})\preceq (r_{1},r_{2},\dots, r_{2n})$ if and only if $p_{i}\leq r_{i}$ for each $i$, in the standard ordering $\leq$ on $\N$). Let $T \subseteq S$ be the inverse image under $f$ of the set of minimal points in $f(S)$. Since $f$ is clearly injective, $T$ is finite, and so we can choose $(\nu(a),\nu(b)) \in T$ such that $|\nu(a)|/|\nu(b)|$ is maximal. (Note that, by Lemma~\ref{unit-canc-lem}, if $c \in \langle X \rangle$ is such that $c=_M 1$, then $c = 1$. So $|\nu(a)| \neq 0 \neq |\nu(b)|$, for all $(\nu(a),\nu(b)) \in S$.) We claim that
\begin{equation} \tag{\dag} \label{eq1}
\frac{|\nu(c)|}{|\nu(d)|} \leq \frac{|\nu(a)|}{|\nu(b)|},
\end{equation}
for all $(\nu(c),\nu(d)) \in S$.

To prove the claim we proceed by induction on $|\nu(c)|+|\nu(d)|$, for $(\nu(c),\nu(d)) \in S$. There are no $(\nu(c),\nu(d)) \in S$ satisfying $|\nu(c)|+|\nu(d)| = 1$, and so~\eqref{eq1} holds vacuously in that case. Suppose inductively that there exists $r \in \N^+$ such that~\eqref{eq1} holds for all $(\nu(c),\nu(d)) \in S$ with $|\nu(c)|+|\nu(d)| \leq r$. Let $(\nu(c),\nu(d)) \in S$ be such that 
\[
|\nu(c)|+|\nu(d)| = r+1.
\] 
If $(\nu(c),\nu(d)) \in T$, then~\eqref{eq1} is satisfied, by the choice of $(\nu(a),\nu(b))$, and so we may assume that $(\nu(c),\nu(d)) \notin T$. Write $\nu(c) = x_{1}^{k_1} \cdots x_{n}^{k_n}$ and $\nu(d) = x_{1}^{m_1} \cdots x_{n}^{m_n}$, for appropriate $k_i, m_i \in \N$. Then there exists $(\nu(e),\nu(f)) \in T$, where $\nu(e) = x_{1}^{p_1} \cdots x_{n}^{p_n}$ and $\nu(f) = x_{1}^{q_1} \cdots x_{n}^{q_n}$, for some $p_i, q_i \in \N$ satisfying $p_i \leq k_i$, $q_i \leq m_i$ for each $i \in \{1, \dots, n\}$, with at least one of those inequalities strict. Thus
\[
|\nu(e)|+|\nu(f)| < |\nu(c)|+|\nu(d)|.
\]
Applying Lemma~\ref{duo-lem}(1) repeatedly, we can write $\nu(c) =_M g\nu(e)$ and $\nu(d) =_M h\nu(f)$, for some $g,h \in \langle X \rangle$ such that $|\nu(c)| = |g| + |\nu(e)|$, $|\nu(d)| = |h| + |\nu(f)|$, and at least one of $g$ and $h$ is different from $1$. As an element of $S$, $(\nu(e),\nu(f)) \neq (1,1)$, and so
\[
|\nu(g)|+|\nu(h)| = |g|+|h| < |\nu(c)|+|\nu(d)| = r+1.
\]
Since 
\[
\nu(g)\nu(e) =_M g\nu(e) =_M \nu(c) =_M \nu(d) =_M h\nu(f) =_M \nu(h)\nu(f),
\]
$\nu(e) =_M \nu(f)$, and $M$ is cancellative, we conclude that $\nu(g) =_M \nu(h)$. Since at least one of $g$ and $h$ is different from $1$, we have $(\nu(g),\nu(h)) \in S$. Hence
\[
\frac{|\nu(g)|}{|\nu(h)|} \leq \frac{|\nu(a)|}{|\nu(b)|},
\]
by the inductive hypothesis. Therefore
\[
\frac{|\nu(c)|}{|\nu(d)|} = \frac{|\nu(g)| + |\nu(e)|}{|\nu(h)| + |\nu(f)|} \leq \frac{|\nu(a)|}{|\nu(b)|},
\]
as claimed.

Finally, for all $c,d \in \langle X \rangle \setminus \{1\}$ satisfying $c=_M d$, we have
\[
\frac{|c|}{|d|} = \frac{|\nu(c)|}{|\nu(d)|} \leq \frac{|\nu(a)|}{|\nu(b)|} = \frac{|a|}{|b|} .
\]
It follows that
\[
\rho(\mathsf{L}_M(a)) = \frac{\max (\mathsf{L}_M(a))}{\min (\mathsf{L}_M(a))} = \frac{|a|}{|b|},
\]
and $\rho(\mathsf{L}_M(c)) \leq \rho(\mathsf{L}_M(a))$ for all $c \in \langle X \rangle$. Hence $\rho(\mathsf{L}_M(a)) = \rho(M)$, i.e., $M$ has accepted elasticity.
\end{proof}

We are now ready for our main result.

\begin{theorem} \label{duo-thrm}
Let $M := \langle X \mid R \rangle$, where $X$ is finite but nonempty, and $\, \{|a|-|b| \mid (a,b)\in R\}$ is finite. If $M$ is normalizing and cancellative, then it satisfies the Structure Theorem for Unions. If $M$ is additionally BF, then it satisfies the Strong Structure Theorem for Unions.
\end{theorem}

\begin{proof}
Since $\{|a|-|b| \mid (a,b)\in R\}$ is finite, so is $\Delta (M)$, by Corollary~\ref{fin-delta}.

If $M$ is not BF, then $\rho_n(M) = \infty$ for some $n \in \N^+$, and hence $M$ satisfies the Structure Theorem for Unions, by Proposition~\ref{tringali-prop}(1). So we may assume that $M$ is normalizing, cancellative, and BF. Then $M$ has accepted elasticity, by Proposition~\ref{duo-prop}, and therefore satisfies the strong Structure Theorem for Unions, by Proposition~\ref{tringali-prop}(3).
\end{proof}

Next we show that the monoids in Proposition~\ref{duo-prop} and Theorem~\ref{duo-thrm} are atomic.

\begin{proposition} \label{norm-atomic-prop}
Let $M$ be a normalizing unit-cancellative monoid. 
\begin{enumerate}[\rm (1)]
\item $M$ is acyclic.

\item If $M$ is finitely-generated, then it is atomic.

\item If $M$ is finitely-generated and reduced, then $\mathcal{A}(M)$ is the unique $($necessarily finite$)$ irredundant generating set for $M$.
\end{enumerate}
\end{proposition}

\begin{proof}
(1) Suppose that $a,b,c \in M$ satisfy $a=bac$. Then $a=adc=bea$, for some $d,e\in M$, since $M$ is normalizing. Given that $M$ is unit-cancellative, we conclude that $dc,be \in M^{\times}$. Since $M$ is normalizing, it follows that $b,c \in M^{\times}$. Hence $M$ is acyclic.

\smallskip

(2) According to~\cite[Theorem 5.8]{Co-Tr23a}, every finitely-generated normalizing monoid satisfies the ascending chain condition on principal ideals. Moreover, according to~\cite[Corollary 2.6]{Tr23a}, an acyclic monoid $M$ is atomic if and only if $M$ has a generating set $X$ such that every ascending chain of principal ideals, starting with $MxM$ for some $x \in X$, stabilizes. Therefore the desired conclusion follows from (1).

\smallskip

(3) Suppose that $M$ is finitely-generated and reduced. Then it is atomic, by (2). Since $M$ is finitely-generated, we can find an irredundant generating set $X$ for $M$. Since $M$ is reduced, necessarily $X = \mathcal{A}(M)$, by Proposition~\ref{atomic-prop}(3). Since $X$ was an arbitrary irredundant generating set, and since $M$ is finitely-generated, $X=\mathcal{A}(M)$ must be finite. 
\end{proof}

It is necessary to assume that the monoid in Proposition~\ref{norm-atomic-prop} is normalizing, for the conclusions to hold. For example, $\langle x, y \mid (y, xyx) \rangle$ is clearly not acyclic, and it is not atomic, since $y$ is neither a unit nor a product of atoms. It is not hard to see, however, that this monoid is cancellative (e.g., using Proposition~\ref{adyan-prop} below).

It can be shown that every monoid which is atomic, finitely-generated, commutative, and unit-cancellative has finite elasticity \cite[Corollary 3.4]{CT} (in terms of atoms), but it need not be accepted \cite[Remark 3.11(5)]{FGKT}. The example constructed in \cite[Remark 3.11(5)]{FGKT} is reduced, and so, in view of Proposition~\ref{norm-atomic-prop}(3), we cannot replace ``cancellative" with ``unit-cancellative" in the hypotheses of Proposition~\ref{duo-prop}. We discuss in Section~\ref{eg-sect} the role of the normalizing assumption in Proposition~\ref{duo-prop} and Theorem~\ref{duo-thrm}.

\section{Examples} \label{eg-sect}

In this section we construct finitely-presented cancellative monoids with unruly factorization behaviors, which demonstrate the sharpness of our results in Section~\ref{duo-sect}. Showing that these monoids are cancellative can be accomplished most conveniently by using a classical result, which we record next.

Given a monoid $\langle X \mid R \rangle$, where $(a,1), (1,a) \notin R$ for all $a \in \langle X \rangle$, denote by $G_l(X \mid R)$, respectively $G_r(X \mid R)$, the undirected graph having vertex set $X$ and edge set $R$, such that for each $(a,b) \in R$ there is an edge between the initial letters of $a$ and $b$, respectively between the final letters of $a$ and $b$. A monoid $M$ is called \emph{Adyan} if $M = \langle X \mid R \rangle$, where $X$ and $R$ are finite, $(a,1), (1,a) \notin R$ for all $a \in \langle X \rangle$, and $G_l(X \mid R)$ and $G_r(X \mid R)$ are acyclic. 

\begin{proposition}[Adyan~\cite{A}] \label{adyan-prop}
Every Adyan monoid embeds in a group, and is therefore cancellative.
\end{proposition}

In particular, a monoid of the form $\langle X \mid (a,b) \rangle$ is cancellative, provided that $a,b \in \langle X \rangle \setminus \{1\}$, and $a$ and $b$ begin, as well as end, with different elements of $X$. We note that our description of Adyan's result is based on that given in~\cite{R}, which also contains a more general criterion, though in that publication it is stated for semigroups rather than monoids. See~\cite{D} for another cancellativity criterion for monoids with presentations.

Our first construction is of a finitely-presented cancellative BF monoid that does not have accepted elasticity (and is not normalizing), showing the necessity of assuming that the monoid in Proposition~\ref{duo-prop} is normalizing, for the conclusion of that result to hold. We begin with a technical lemma.

\begin{lemma} \label{eg1-lem}
Let $M := \langle x,y,z \mid (xy,yzx) \rangle$.
\begin{enumerate}[\rm (1)]
\item For all $k,l\in \N^+$, we have $\, \mathsf{L}_M(x^ky^l) = [k+l, 2(k+l)-1]$.

\item For all $k,l\in \N^+$, we have  $\rho(\mathsf{L}_M(x^ky^l)) = 2-\frac{1}{k+l}$.

\item For all $n\in \N$, among $a \in \langle x,y,z \rangle$ such that $n = |a|$, $\rho(\mathsf{L}_M(a))$ is maximal for any $a = x^ky^l$ with $n=k+l$.
\end{enumerate}
\end{lemma}

\begin{proof}
(1) Let $k,l \in \N^+$. Since $x^ky^l$ contains no instances of $z$, no application of the defining relation to this word in can result in a shorter one. Now,
\[x^ky^l = x^{k-1}(xy)y^{l-1} =_M x^{k-1}(yzx)y^{l-1},\]
where the last term has length $k+l+1$ (which proves the claim for $k=1=l$, since no other expansions are possible in that case). Next, the sequence
\[x^{k-1}(yzx)y^{l-1} = x^{k-2}(xy)zxy^{l-1} =_M x^{k-2}(yzx)zxy^{l-1} = x^{k-2}y(zx)z(xy)y^{l-2}\]
\[=_M x^{k-2}y(zx)z(yzx)y^{l-2} = x^{k-2}y(zx)z(yz)xy^{l-2}\]
contains words of length $k+l+2$ and $k+l+3$ (provided that both $k \geq 2$ and $l \geq 2$), and this proves the claim for $k,l\leq 2$. Continuing in this fashion,
\[x^{k-2}y(zx)z(yz)xy^{l-2} =_M \cdots =_M x^{k-3}y(zx)^2z(yz)^2xy^{l-3} =_M \cdots =_M y(zx)^{k-1}z(yz)^{l-1}x,\]
and the sequence contains words of all lengths in $[k+l, 2(k+l)-1]$. Since $y(zx)^{k-1}z(yz)^{l-1}x$ contains no instances of $xy$, no application of the defining relation of $M$ to that word can result in a longer one. Moreover, the above sequence contains all the possible applications of the defining relation to a word obtained from $x^ky^l$, and so $\mathsf{L}_M(x^ky^l) = [k+l, 2(k+l)-1]$. 

\smallskip

(2) By (1), for all $k,l\in \N^+$, we have
\[\rho(\mathsf{L}_M(x^ky^l)) = \frac{\max (\mathsf{L}_M(x^ky^l))}{\min (\mathsf{L}_M(x^ky^l))} = \frac{2(k+l)-1}{k+l} = 2-\frac{1}{k+l}.\]

\smallskip

(3) Let $n\in \N$, let $a \in \langle x,y,z \rangle$ be such that $n = |a|$, and suppose that $\rho(\mathsf{L}_M(a))$ is maximal among elements of length $n$. If $n=0$, then $a=1=x^0y^0$, and $\rho(\mathsf{L}_M(a)) = 0$. If $n = 1$, then $a \in \{x,y,z\}$, and $\rho(\mathsf{L}_M(a)) = 1$, regardless of the choice of $a$. So we may assume that $n \geq 2$.

Next, suppose that $a = (uy)(xv)$ for some $u,v \in \langle x,y,z \rangle$, and let $b := uxyv$ (so $|a|=|b|$). Then
\[b =_M u(yzx)v =_M (uy)z(xv).\]
If either $u$ cannot be expanded to a word that ends in $x$, or $v$ cannot be expanded to a word that ends in $y$, then any applications of the defining relation to $a$, that results in a longer word, can also be applied to $(uy)z(xv)$, which contradicts the choice of $a$. On the other hand, if $u =_M u'x$ and $v =_M yv'$ for some $u', v' \in \langle x,y,z \rangle$, then 
\[a =_M u'(yzx)(yzx)v' =_M u'(yz)(yzx)(zx)v' =_M u'y(zyzxz)xv',\] 
while 
\[b =_M u'(xy)z(xy)v' =_M u'(yzx)z(yzx)v' =_M u'y(zxzyz)xv',\]
and the two admit the same expansions thereafter. Therefore we may assume that $a$ does not contain $yx$ as a subword. 

Now, in order to be expandable into a longer word, $a$ must contain at least one instance of $xy$. Suppose that $a = uzx^kyv$, for some $u,v \in \langle x,y,z \rangle$ and $k \in \N^+$, and let $b := ux^{k+1}yv$. Then any applications of the defining relation to $a$, that results in a longer word, can also be applied to $b$, but $b$ is further expandable, by (1). So we may assume that $a = x^kyv$, for some $v \in \langle x,y,z \rangle$ and $k \in \N^+$. Similarly, $a$ cannot be of the form  $x^ky^lzv$, for some $v \in \langle x,y,z \rangle$ and $k,l \in \N^+$, since $x^ky^{l+1}v$ is strictly more expandable. Therefore $a$ can be taken to be of the form $x^ky^l$, for some $k,l \in \N$ satisfying $n=k+l$.
\end{proof}

\begin{proposition} \label{elasticity-eg}
The monoid $M := \langle x,y,z \mid (xy,yzx) \rangle$ is cancellative, and $\rho(M) = 2$ $($so $M$ is BF$)$, but $M$ does not have accepted elasticity.
\end{proposition}

\begin{proof}
First, note that $\rho(\mathsf{L}_M(1)) = 0$, and $\rho(\mathsf{L}_M(x)) = 1$. Hence, by Lemma~\ref{eg1-lem}(2,3), we have 
\[\rho(M) = \sup_{n \in \N^+} \bigg(2-\frac{1}{n}\bigg) = 2,\]
but $\rho(L) < 2$ for all $L \in \mathcal{L}(M)$. Hence $M$ does not have accepted elasticity. The conclusion that $M$ is cancellative follows from Proposition~\ref{adyan-prop} (see subsequent comment).
\end{proof}

While the monoid in the previous proposition does not have accepted elasticity, it does satisfy the Structure Theorem for Unions, by Proposition~\ref{one-relation-prop}. Our next goal is to construct a finitely-presented cancellative monoid that does not satisfy the Structure Theorem for Unions (and is not normalizing), to demonstrate the necessity of the normalizing assumption in Theorem~\ref{duo-thrm}. 

Our construction requires a couple of lemmas, though the last statement in the following lemma is included for its own interest.

\begin{lemma} \label{eg-lem}
Let $n > 1$, and $M := \langle x,y \mid (xy, y^nx) \rangle$.
\begin{enumerate}[\rm (1)]
\item $M$ is cancellative.

\item For all $k \in \N$, we have $x^{k}y =_M y^{n^k}x^k$.

\item For all $k\in \N^+$, we have $\rho_k(M) = n^{k-1}+k-1$.

\item The differences $\rho_k(M) - \rho_{k-1}(M)$ are finite but grow without bound. 
\end{enumerate}
\end{lemma}

\begin{proof}
(1) This follows from Proposition~\ref{adyan-prop}.

\smallskip

(2) The equation in the statement clearly holds for $k=0$. Assume inductively that $x^{k}y =_M y^{n^k}x^k$, for some $k\geq 0$. Then
\[x^{k+1}y = x(x^{k}y) =_M x(y^{n^k}x^k) =_M (y^{n})^{n^k}x^{k+1} =_M y^{n^{k+1}}x^{k+1},\]
via $n^k$ applications of the defining relation $xy =_M y^nx$, giving the desired conclusion.

\smallskip

(3) Since $n > 1$, for any $u,v \in \langle x,y\rangle$, the element $uxyv \, (=_M uy^nxv)$ has a factorization of greater length than $uyxv$. It follows that among the elements of $\langle x,y\rangle$ of length $k \in \N^+$, $x^{k-1}y$ has a factorization of greatest possible length. Since $y^{n^{k-1}}x^{k-1}$ admits no further expansions, the longest possible factorization of $x^{k-1}y$ has length $n^{k-1}+ k-1$, by (2). Hence $\rho_k(M) = n^{k-1}+k-1$.

\smallskip

(4) By (3), for any $k \geq 2$ we have
\[\rho_k(M) - \rho_{k-1}(M) = n^{k-1}+k-1 - (n^{k-2}+k-2) =  n^{k-2}(n-1)+1,\]
from which the statement follows.
\end{proof}

\begin{lemma} \label{eg3-lem}
Let $n > 1$, and $M := \langle u,v,x,y \mid (u^2,v^3), (xy, y^nx) \rangle$.
\begin{enumerate}[\rm (1)]
\item For all $k\in \N^+$, we have $\rho_k(M) = n^{k-1}+k-1$.

\item For each $k \geq 3$, there is a maximal $\mu_k(M) \in \N^+$ such that $\, \{\mu_k(M)-1, \mu_k(M)\} \subseteq \mathcal{U}_{k}(M)$, and, moreover, $\mu_{k}(M) \leq \rho_{k-2}(M)+4$.
\end{enumerate}
\end{lemma}

\begin{proof}
(1) For any $w_1,w_2 \in \langle u,v,x,y\rangle$, the element $w_1xyw_2 \, (=_M w_1y^nxw_2)$ has a factorization of greater length than $w_1yxw_2$. Also, given any element of $\langle u,v,x,y\rangle$ containing at least two (not necessarily distinct) elements of $\{u,v\}$, one can obtain an element of $\langle u,v,x,y\rangle$ with a factorization of greater length by removing the two instances of $u$ and $v$, and adding $xy$ at the end. Finally, given an element of $\langle u,v,x,y\rangle$ containing $y$ and at least one instance of $u$ or $v$, one can obtain an element of $\langle u,v,x,y\rangle$ with a factorization of greater length by removing a $u$ or $v$, and adding an $x$ immediately to the left of the $y$. It follows that among elements of $\langle u,v,x,y\rangle$ of length $k \in \N^+$, $x^{k-1}y$ has a factorization of greatest possible length. Since $y^{n^{k-1}}x^{k-1}$ admits no further expansions, the longest possible factorization of $x^{k-1}y$ has length $n^{k-1}+ k-1$, by Lemma~\ref{eg-lem}(2), from which the claim follows.

\smallskip

(2) Let $k \geq 3$. Then $u^k = u^2u^{k-2} =_M v^3 u^{k-2}$, and so $\{k,k+1\} \subseteq \mathcal{U}_{k}(M)$. Since, by (1), $\mathcal{U}_{k}(M)$ is finite, the quantity $\mu_k(M)$ is well-defined.

Now, let $a \in \langle u,v,x,y\rangle$ be such that $|a|=k$. If $a$ contains more than one instance of $u$ and $v$, then, in view of the defining relations, $\max(\mathsf{L}_M(a)) \leq \rho_{k-2}(M) + 3$. Thus if $m \in \mathcal{U}_{k}(M)$ is such that $m\geq \rho_{k-2}(M)+4$, then the difference between $m$ and the successive element of $\mathcal{U}_{k}(M)$ (if there is one) must be at least $2$, since it must be the result of applying repeatedly the relation $(xy,y^nx)$ to a shorter word. It follows that $\mu_k(M) \leq \rho_{k-2}(M)+4$.
\end{proof}

\begin{proposition} \label{no-structure-eg}
For any $n>1$, the monoid $M := \langle u,v,x,y \mid (u^2,v^3), (xy, y^nx) \rangle$ is cancellative and BF, but does not satisfy the Structure Theorem for Unions.
\end{proposition}

\begin{proof}
By Lemma~\ref{eg3-lem}(2), for each $k \geq 3$, we have $\{\mu_k(M)-1, \mu_k(M)\} \subseteq \mathcal{U}_{k}(M)$. So if $\mathcal{U}_k (M) \subseteq k + d\Z$ for some $d \in \N^+$ and $k \geq 3$, then necessarily $d=1$. Now, suppose that there exist $k^*\in \N^+$ and $m\in \N$, such that for all $k \geq k^*$,
\[
(k + \Z) \cap [\lambda_{k}(M) + m, \rho_{k}(M) - m] \subseteq \mathcal{U}_k (M),
\]
and hence 
\[
[\lambda_{k}(M) + m, \rho_{k}(M) - m] \subseteq \mathcal{U}_k (M).
\]
Then, by Lemma~\ref{eg3-lem}(2), for all sufficiently large $k$, we have $\rho_{k}(M) - m \leq \rho_{k-2}(M)+4$, and therefore, using Lemma~\ref{eg3-lem}(1) gives
\[
m+4 \geq \rho_{k}(M) - \rho_{k-2}(M) = (n^{k-1}+k-1) - (n^{k-3}+k-3) = n^{k-3}(n^2-1) + 2.
\]
Clearly, no $m\in \N$ has this property, and so $M$ does not satisfy the Structure Theorem for Unions. Lemma~\ref{eg3-lem}(1) also implies that $M$ is BF.

The conclusion that $M$ is cancellative follows easily from Proposition~\ref{adyan-prop}.
\end{proof}

\section{Summary}

It is well-known that commutative finitely-generated cancellative monoids have accepted elasticity \cite[Theorem 3.1.4]{GHK}, satisfy the Structure Theorem for Unions \cite[Theorem 3.6]{FGKT}, and are often fully elastic (e.g., \cite[Theorem 1.2]{Zh19a}), in terms of factorization into atoms. In this paper we prove analogous statements for certain non-commutative monoids, in the broader context of factorization into generators (Proposition~\ref{duo-prop}, and Theorems~\ref{duo-thrm} and~\ref{fully-elastic-theorem}), and give examples showing that, even for finitely-presented cancellative monoids, these properties need not always hold (Propositions~\ref{elasticity-eg} and~\ref{no-structure-eg}). Much remains to be explored, however, as summarized in the following general problem.

\begin{problem}
Let $M := \langle X \mid R \rangle$, where $X$ and $R$ are finite. Characterize $X$ and $R$ for which $M$ is fully elastic, $M$ has accepted elasticity, and $M$ satisfies the Structure Theorem for Unions. 
\end{problem}

\section*{Acknowledgements} 

This work was supported by the Austrian Science Fund FWF, Project P36852-N. Furthermore, the authors would like to thank Salvatore Tringali for fruitful discussions, and the anonymous referee for pointing us to relevant literature and providing helpful suggestions.

\bigskip

\noindent NAWI Graz, Department of Mathematics and Scientific Computing, University of Graz, Heinrichstraße 36, 8010 Graz, Austria

\noindent \emph{Email:} \href{mailto:alfred.geroldinger@uni-graz.at}{alfred.geroldinger@uni-graz.at}

\bigskip

\noindent Department of Mathematics, University of Colorado, Colorado Springs, CO, 80918, USA 

\noindent \emph{Email:} \href{mailto:zmesyan@uccs.edu}{zmesyan@uccs.edu}


\begin{thebibliography}{00}

\bibitem{A} S.~I.~Adyan, \textit{On the Embeddability of Semigroups in Groups,} Sov.\ Math., Dokl.\ \textbf{1} (1960) 819--821.

\bibitem{AM} E.~Akalan and H.~Marubayashi, \textit{Multiplicative Ideal Theory in Non-Commutative Rings}, in \textit{Multiplicative Ideal Theory and Factorization Theory,} Springer Proceedings in Mathematics \& Statistics \textbf{170}, Springer, Cham, 2016, pp.~1--21.

\bibitem{A-C-H-Z07}
D.~D.~Anderson, J.~Coykendall, L.~Hill, and M.~Zafrullah, \textit{Monoid Domain Constructions of Antimatter Domains,} Comm.\ Algebra \textbf{35} (2007) 3236--3241.

\bibitem{BS} N.~R.~Baeth and D.~Smertnig, \textit{Factorization Theory: From Commutative
to Noncommutative Settings,} J.\ Algebra \textbf{441} (2015) 475--551.

\bibitem{Ba-Po25a} A.~Bashir and M.~Pompili, \textit{On Transfer Homomorphisms in Commutative Rings with Zero Divisors,} Comm.\ Algebra \textbf{54} (2026) 776--78.

\bibitem{BBNS} J.~P.~Bell, K.~Brown, Z.~Nazemian, and D.~Smertnig, \textit{On Noncommutative Bounded Factorization Domains and Prime Rings,} J.\ Algebra \textbf{622} (2023) 404--449.

\bibitem{Be-El23a} M.~Benelmekki and S.~El Baghdadi, \textit{When is a Group Algebra Antimatter,} in \textit{Algebraic, Number Theoretic, and Topological Aspects of Ring Theory,} Springer, Cham, 2023, pp.~87--98.

\bibitem{Bo-Co-Kh26a} N.~Bogdanovic, L.~Cossu, and M.~A.~Khadam, \textit{Atoms in Monoids of Ideals,} preprint (arXiv: 2510.24455).

\bibitem{MR297758} A.~Bouvier, \textit{Sur les anneaux de fractions des anneaux atomiques pr\'{e}simplifiables,} Bull.\ Sci.\ Math.\ \textbf{95} (1971) 371--377.

\bibitem{MR297745} A.~Bouvier, \textit{Anneaux pr\'{e}simplifiables,} C.\ R.\ Acad.\ Sci.\ Paris S\'{e}r.\ A-B \textbf{274} (1972) A1605--A1607.

\bibitem{Ch-Re20a} G.~W.~Chang and A.~Reinhart, \textit{Unique Factorization Properties of Non-Unique Factorization Domains II,} J.\ Pure Appl.\ Algebra \textbf{224} (2020) 106430, 18 pp.

\bibitem{Ch-Re26a} G.~W.~Chang and A.~Reinhart, \textit{Valuation Ideal Factorization Domains,} J.\ Commut.\ Algebra (2026), in press.

\bibitem{CDHK} S.~T.~Chapman, J.~Daigle, R.~Hoyer, and N.~Kaplan, \textit{Delta Sets of Numerical Monoids Using Nonminimal Sets of Generators,} Comm.\ Algebra \textbf{38} (2010) 2622--2634.

\bibitem{C-G-L-M-S12} S.~T.~Chapman, P.~A.~Garc{\'i}a-S{\'a}nchez, D.~Llena, A.~Malyshev, and D.~Steinberg, \textit{On the Delta Set and The Betti Elements of a BF-monoid}, Arab.\ J.\ Sci.\ Eng.\ \textbf{1} (2012) 53--61.

\bibitem{CGLPR} S.~T.~Chapman, P.~A.~Garc{\'i}a-S{\'a}nchez, D.~Llena, V.~Ponomarenko, and J.~C.~Rosales, \textit{The Catenary and Tame Degree in Finitely Generated Commutative Cancellative Monoids,} Manuscr.\ Math.\ \textbf{120} (2006) 253--264.

\bibitem{CP} A.~H.~Clifford and G.~B.~Preston, \textit{The Algebraic Theory of Semigroups, Volume I,} Second Edition, Providence, Rhode Island, 1964.

\bibitem{Co25a} L.~Cossu, \textit{Some Applications of a New Approach to Factorization}, in \textit{Recent Progress in Ring and Factorization Theory,} Springer Proceedings in Mathematics \& Statistics \textbf{477}, Springer, Cham, 2025, pp.~73--94.

\bibitem{Co-Tr23a} L.~Cossu and S.~Tringali, \textit{Factorization Under Local Finiteness Conditions,} J.\ Algebra \textbf{630} (2023) 128--161.

\bibitem{CT} L.~Cossu and S.~Tringali, \textit{On the Finiteness of Certain Factorization Invariants,} Ark.\ Mat.\ \textbf{62} (2024) 21--38.

\bibitem{Co-Tr23b} L.~Cossu and S.~Tringali, \textit{Abstract Factorization Theorems with Applications to Idempotent Factorizations,} Israel J.\ Math.\ \textbf{263} (2024) 349--395.

\bibitem{Co-Tr24b} L.~Cossu and S.~Tringali, \textit{On the Arithmetic of Power Monoids,} J.\ Algebra \textbf{686} (2026) 793--813.

\bibitem{Co-Go25a} J.~Coykendall and F.~Gotti, \textit{Atomicity in Integral Domains}, in \textit{Recent Progress in Ring and Factorization Theory,} Springer Proceedings in Mathematics \& Statistics \textbf{477}, Springer, Cham, 2025, pp.~95--159.

\bibitem{D} P.~Dehornoy, \textit{A Cancellativity Criterion for Presented Monoids,} Semigroup Forum \textbf{99} (2019) 368--390.

\bibitem{FGKT} Y.~Fan, A.~Geroldinger, F.~Kainrath, and S.~Tringali, \textit{Arithmetic of Commutative Semigroups with a Focus on Semigroups of Ideals and Modules,} J.\ Algebra Appl.\ \textbf{16} (2017) 1750234, 42 pp.

\bibitem{Fa-Tr18a} Y.~Fan and S.~Tringali, \textit{Power Monoids: A Bridge Between Factorization Theory and Arithmetic Combinatorics,} J.\ Algebra \textbf{512} (2018) 252--294.
  
\bibitem{Ga-Li11b} W.~Gao and Y.~Li, \textit{On Duo Group Rings}, Algebra Colloq.\ \textbf{18} (2011) 163--170.  

\bibitem{GS25b} P.~A.~Garc{\'i}a-S{\'a}nchez, \textit{Factorizations Into Irreducible Numerical Semigroups,} Commun.\ Korean Math.\ Soc.\ \textbf{40} (2025) 587--592.
  
\bibitem{Ga-Oj-SN13a} P.~A.~Garc{\'i}a-S{\'a}nchez, I.~Ojeda, and A.~S{\'a}nchez-R.-Navarro, \textit{Factorization Invariants in Half-Factorial Affine Semigroups}, Internat.\ J.\ Algebra Comput.\ \textbf{23} (2013) 111--122.  
  
\bibitem{GS-Tr25a} P.~A.~Garc{\'i}a-S{\'a}nchez and S.~Tringali, \textit{Semigroups of Ideals and Isomorphism Problems,} Proc.\ Amer.\ Math.\ Soc.\ \textbf{153} (2025) 2323--2339.

\bibitem{G0} A.~Geroldinger, \textit{Non-Commutative Krull Monoids: A Divisor Theoretic Approach and Their Arithmetic,} Osaka J.\ Math.\ \textbf{50} (2013) 503--539.

\bibitem{G} A.~Geroldinger, \textit{Sets of Lengths,} Amer.\ Math.\ Monthly \textbf{123} (2016) 960--988.

\bibitem{GHK} A.~Geroldinger and F.~Halter-Koch, \textit{Non-Unique Factorizations: Algebraic, Combinatorial and Analytic Theory,} Pure Appl.\ Math.\ \textbf{278}, Chapman \& Hall/CRC, 2006.

\bibitem{GK} A.~Geroldinger and M.~A.~Khadam, \textit{On the Arithmetic of Monoids of Ideals,} Ark.\ Mat.\ \textbf{60} (2022) 67--106.

\bibitem{Ge-Lo-Ki26} A.~Geroldinger, H.~Kim, and K.~A.~Loper, \textit{On Long-Term Problems in Multiplicative Ideal Theory and Factorization Theory,} in \textit{The Ideal Theory and Arithmetic of Rings, Monoids, and Semigroups,} Contemp.\ Math., Amer.\ Math.\ Soc., Providence, RI, in press.

\bibitem{Ge-Sc-Zh17b} A.~Geroldinger, W.~A.~Schmid, and Q.~Zhong, \textit{Systems of Sets of Lengths: Transfer Krull Monoids Versus Weakly Krull Monoids,} in \textit{Rings, Polynomials, and Modules,} Springer, Cham, 2017, pp.~191--235.

\bibitem{GS} A.~Geroldinger and E.~D.~Schwab, \textit{Sets of Lengths in Atomic Unit-Cancellative Finitely Presented Monoids,} Colloq.\ Math.\ \textbf{151} (2018) 171--187.

\bibitem{Ge-Zh20a} A.~Geroldinger and Q.~Zhong, \textit{Factorization Theory in Commutative Monoids,} Semigroup Forum \textbf{100} (2020) 22--51.

\bibitem{H-K} F.~Halter-Koch, \textit{Über Längen nicht-eindeutiger Faktorisierungen und Systeme linearer diophantischer Ungleichungen,} Abh.\ Math.\ Sem.\ Univ.\ Hamburg \textbf{63} (1993) 265--276.

\bibitem{H} J.~M.~Howie, \textit{Fundamentals of Semigroup Theory,} Oxford University Press, Oxford-New York, 1995.

\bibitem{Iz-Kn26a} Z.~Izhakian and M.~Knebusch, \textit{Supertropical Monoids III: Factorization and Splitting Covers,} J.\ Algebra Appl.\ \textbf{25} (2026) 2650015.

\bibitem{JO} E.~Jespers and J.~Okniński, \textit{Noetherian Semigroup Algebras,} Algebra and Applications \textbf{7}, Springer, Dordrecht, 2007.

\bibitem{Li-Sc-Si25} Z.~Li, H.P.~Schr\"ocker, and J.~Siegele, \textit{A Geometric Algorithm for the Factorization of Rational Motions in Conformal Three Space,} J.\ Symbolic Comput.\ \textbf{128} (2025) 102388, 20 pp.

\bibitem{NB} C.~F.~Nyberg-Brodda, \textit{The Word Problem for One-Relation Monoids: A Survey,} Semigroup Forum \textbf{103} (2021) 297--355.

\bibitem{O} J.~Okniński, \textit{Noetherian Semigroup Algebras and Beyond}, in \textit{Multiplicative Ideal Theory and Factorization Theory,} Springer Proceedings in Mathematics \& Statistics \textbf{170}, Springer, Cham, 2016, pp.~255--276.

\bibitem{P1} A.~Philipp, \textit{A Characterization of Arithmetical Invariants by the Monoid of Relations,} Semigroup Forum \textbf{81} (2010) 424--434.

\bibitem{P2} A.~Philipp, \textit{A Characterization of Arithmetical Invariants by the Monoid of Relations {II}: The Monotone Catenary Degree and Applications to Semigroup Rings}, Semigroup Forum \textbf{90} (2015) 220--250.

\bibitem{Ra25b} B.~Rago, \textit{A Characterization of Transfer Krull Orders in Dedekind Domains with Torsion Class Group,} Canad.\ Math.\ Bull., in press.

\bibitem{Re26a} A.~Reinhart, \textit{On the System of Length Sets of Power Monoids,} preprint (arXiv: 2508.10209).
  
\bibitem{R} J.~H.~Remmers, \textit{On the Geometry of Semigroup Presentations,} Adv.\ Math.\ \textbf{36} (1980) 283--296.

\bibitem{Ro-GS-GG04b} J.~C.~Rosales, P.~A.~Garc{\'i}a-S{\'a}nchez, and J.~I.~Garc{\'i}a-Garc{\'i}a, \textit{Atomic Commutative Monoids and Their Elasticity,} Semigroup Forum \textbf{68} (2004) 64--86.

\bibitem{Sm13a} D.~Smertnig, \textit{Sets of Lengths in Maximal Orders in Central Simple Algebras,} J.\ Algebra \textbf{390} (2013) 1--43.

\bibitem{S} D.~Smertnig, \textit{Factorizations of Elements in Noncommutative Rings: A Survey,} in \textit{Multiplicative Ideal Theory and Factorization Theory,} Springer Proceedings in Mathematics \& Statistics \textbf{170}, Springer, Cham, 2016, pp.~353--402.

\bibitem{Sm19a} D.~Smertnig, \textit{Factorizations in Bounded Hereditary Noetherian Prime Rings,} Proc.\ Edinburgh Math.\ Soc.\ \textbf{62} (2019) 395--442.

\bibitem{T} S.~Tringali, \textit{Structural Properties of Subadditive Families with Applications to Factorization Theory,} Israel J.\ Math.\ \textbf{234} (2019) 1--35.

\bibitem{Tr22a} S.~Tringali, \textit{An Abstract Factorization Theorem and Some Applications,} J.\ Algebra \textbf{602} (2022) 352--380.

\bibitem{Tr23a} S.~Tringali, \textit{A Characterization of Atomicity}, Math.\ Proc.\ Cambridge Philos.\ Soc.\ \textbf{175} (2023) 459--465.

\bibitem{Zh19a} Q.~Zhong, \textit{On Elasticities of Locally Finitely Generated Monoids}, J.\ Algebra \textbf{534} (2019) 145--167.
  
\end{thebibliography}
\end{document}